\documentclass[11pt,english]{amsart}
\usepackage{amsfonts, amssymb, amsmath, amsthm, latexsym, enumerate, epsfig, color, mathrsfs, fancyhdr, pdfsync, eurosym, endnotes, stmaryrd}
\usepackage[all]{xy}

\usepackage[english]{babel}

\usepackage{raccourcisanglais}

\usepackage[mac]{inputenc}
\usepackage[T1]{fontenc}

\usepackage{hyperref}

\theoremstyle{plain}
\newtheorem{theorem}{Theorem}[section]
\newtheorem{lemma}[theorem]{Lemma}

\newtheorem{corollary}[theorem]{Corollary}
\newtheorem{definition}[theorem]{Definition}

\theoremstyle{remark}
\newtheorem{remark}[theorem]{Remark}

\begin{document}
\setlength{\baselineskip}{0.55cm}	

\title[Polynomial approximation of Berkovich spaces]{Polynomial approximation of Berkovich spaces\\ and definable types}
\author{J\'er\^ome Poineau}
\address{Institut de recherche math\'ematique avanc\'ee, 7, rue Ren\'e Descartes, 67084 Strasbourg, France}
\email{jerome.poineau@math.unistra.fr}
\thanks{The research for this article was partially supported by the ANR project Berko.}

\date{\today}


\begin{abstract}
We investigate affine Berkovich spaces over maximally complete fields and prove that they may be approximated by simpler spaces when the only functions we need to evaluate are polynomials of bounded degree. We derive applications to semi-algebraic sets and recover a result of E.~Hrushovski and F.~Loeser which claims that points of Berkovich spaces give rise to definable types (a model-theoretic notion of tameness). 
\end{abstract}

\maketitle

\section*{Introduction}

In the late eighties, V.~Berkovich came up with a new definition of $p$-adic analytic spaces, and more generally analytic spaces over a complete valued field~$k$. It is remarkable that, in this setting, the affine line~$\Ak$ over~$k$ always contains many points in addition to those belonging to~$k$. This feature, which may seems disturbing at first sight, leads to numerous interesting consequences. For instance, it lets the space~$\Ak$ carry a nice topology, where properties such as path-connectedness and local compactness are satisfied, regarless of the fact that the field~$k$ may be totally disconnected and not locally compact.

On the other hand, due precisely to this abundance of extra points, Berkovich spaces are often difficult to describe explicitely. A noteworthy exception is that of the analytic line~$\Ak$ over an algebraically closed field~$k$. In this situation, V.~Berkovich managed to give a complete classification of the points, which are gathered into four types. The simplest case is that of maximally complete fields, where only three types remain, which all correspond to avatars of generic points of discs (the type being determined by a condition on the radius). 

In this paper, we are interested in the case of a general affine space~$\E{n}{k}$ over a maximally complete field (not necessarily algebraically closed). Let us say at once that we will not give a description which is as explicit as the one provided by V.~Berkovich for the line (we believe such a task to be hopeless in general) but a succession of ``approximations''. More precisely, if we fix a point~$x$ in~$\E{n}{k}$ and a degree~$d$ for polynomials in $n$~variables, we will show how to construct a simple space~$S_{x,d}$ (similar to the disc mentionned before in dimension~1, more complicated in higher dimension) which determines the value of any polynomial of degree at most~$d$ at the point~$x$.

As an application, in the last part of the paper, we prove semi-algebraicity results. We also explain how to recover, in a geometric fashion, the fact, observed by E.~Hrushovski and F.~Loeser, that points of affine Berkovich spaces give rise to definable types (a model-theoretic notion of tameness)\footnote{It is unfortunate that two notions of type appear in this paper. The model-theoretic notion is not to be confused with the one used by V.~Berkovich to classify the points of the line.}.

\medskip

\textbf{Acknowledgements}

In 2010, I was given the opportunity to attend the summer conference of the MRC program on model theory of fields in Snowbird Resort, Utah. I would like to thank all the participants who explained the basics of model theory to me, with much patience and insight.

\section{Approximation in dimension 1}

Let~$(k,|.|)$ be a field endowed with a (possibly trivial) non-Archimedean absolute value, with respect to which it is complete. We will have a special interest in fields~$k$ that are maximally complete. Let us recall that this condition means that the field~$k$ has no non-trivial immediate extension, \textit{i.e.} with the same residue field and the same value group. By a theorem of I.~Kaplansly (see~\cite{maximalfields}), it is equivalent to the spherical completeness of the field: a family of embedded discs cannot have an empty intersection. This last condition is the only one we shall use in this paper. It is easy to check that trivially valued fields and complete discrete valuation fields are maximally complete. For a more complicated example, consider the extension~$F(\!(t^\Q)\!)$ of a field~$F$ given by series $\sum_{q\in \Q} a_{q}\, t^q$ with a well-ordered support. For more on this subject (such as a mixed-characteristic analogue of the previous construction), we advise the reader to have a look at B.~Poonen's nice paper~\cite{Poonenmax}.

\bigskip

In this first section, we will be interested in the affine analytic line~$\E{1}{k}$ over the field~$k$, in the sense of V.~Berkovich. We refer to~\cite{rouge} for basic results concerning Berkovich spaces. Actually, in this paper, we will not need much beyond the first chapter of the book. 

Let us choose some notations. We will denote~$T$ the variable on~$\E{1}{k}$. For any integer~$d$, we denote~$k[T]_{\le d}$ the $k$-vector space of polynomials of degree at most~$d$.

Let~$P$ be an irreducible polynomial. We denote~$\eta_{P,0}$ the unique point in~$\Ak$ such that $P(\eta_{P,0})=0$. Let~$r>0$. The Shilov boundary of the affinoid domain $\{|P|\le r\}$, which coincides with its topological boundary, contains a unique point, that we denote~$\eta_{P,r}$. From a model-theoretic point of view, this is the generic type of the subset $\{|P|\le r\}$.

Let us denote~$D_{0}=\{\eta_{T,1}\}$ and for every integer~$d\ge 1$,
\[D_{d} = \{\eta_{P,r},\ P\in k[T]_{\le d}, r\ge 0\}.\]

\begin{remark}\label{rem:unique}
For any points $x \ne y \in D_{d}$, there exists a polynomial $P\in k[T]_{\le d}$ such that $|P(x)| \ne |P(y)|$.
\end{remark}

Let us add a few words on V.~Berkovich's classifications of the points of~$\Ak$ (see~\cite{rouge}, \S~1.4). When the field~$k$ is algebraically closed and maximally complete, he proves that every point is of the form~$\eta_{T-a,r}$, with $a\in k$ and~$r\ge 0$. We would like to adapt this classification when the condition of being algebraically closed is removed. In this case, polynomials of degree bigger than~1 cannot be decomposed into polynomials of degree~1 and they have to enter the picture too. Actually, we believe that V.~Berkovich's result should, in general, be thought of as a kind of approximation of~$\Ak$, where only polynomials of degree~1 are considered. We observe that his construction may actually be generalized in any degree.

\begin{theorem}\label{thm:dim1}
Assume that~$k$ is maximally complete. Let~$x$ be a point of~$\E{1}{k}$. For any integer $d\ge 0$, there exists a unique point~$x_{d}$ in~$D_{d}$ with the following property:  
\[\forall P\in k[T]_{\le d},\ |P(x_{d})|=|P(x)|.\]

In addition, $x_{d} \ge x$ and
\[x_{d} = \inf(\{y\in D_{d}\, |\, y\ge x\}).\]
\end{theorem}
\begin{proof}
The uniqueness follows from remark~\ref{rem:unique}. Let us assume that we found $x_{d} = \eta_{Q,r}$ with the required properties and prove the second part of the statement. Since~$Q$ has degree at most~$d$, we have $|Q(x)|=|Q(\eta_{Q,r})|=r$. Hence~$x$ belongs to $\{|Q|\le r\}$ and $x\le \eta_{Q,r}=x_{d}$. Now let $y=\eta_{R,s} \in D_{d}$ such that $y\ge x$. We have $s = |R(y)|\ge |R(x)|=|R(x_{d})|$, hence $y\ge x_{d}$ and we are done.

Let us now prove the existence of~$x_{d}$ by induction on~$d$. If $d=0$, the result is obvious. 

Let~$d\ge 0$ be an integer and assume we found~$x_{d}$. If $|P(x_{d})|=|P(x)|$ for any $P\in k[T]_{\le d+1}$, we put $x_{d+1}=x_{d}$ and we are done.

Let us now assume that there exists $A\in k[T]_{\le d+1}$ such that $|A(x)| < |A(x_{d})|$. We may assume that~$A$ is monic of degree~$d+1$.

Let~$k_{d}$ denote the set $\{P(x), P \in k[T]_{\le d}\}$. It is a finite-dimensional normed $k$-vector space. Since~$k$ is spherically complete, $k_{d}$ is spherically complete too (\textit{cf.}~\cite{BGR}, 2.4.4).

For any polynomial~$P$ in $k[T]_{\le d}$, let us consider the ball 
\[E_{P} = \overline{D}(P(T)(x), |(T^{d+1}-P(T))(x)|)\] 
inside~$k_{d}$. Let $P,Q \in k[T]_{\le d}$. If $|(T^{d+1}-P(T))(x)| \le |(T^{d+1}-Q(T))(x)|$, then $|P(T)(x)-Q(T)(x)| \le |(T^{d+1}-Q(T))(x)|$, hence $P(T)(x)\in E_{Q}$ and $E_{P} \subset E_{Q}$.

Since~$k_{d}$ is spherically complete, there exists~$P_{0} \in k[T]_{\le d}$ such that~$P_{0}(T)(x)$ belongs to all the~$E_{P}$'s. Let $R_{0}(T) =  T^{d+1}-P_{0}(T)$ and
\[r= |(R_{0}(T))(x)|= \inf_{P\in k[T]_{\le d}} (|(T^{d+1}-P(T))(x)|).\]
Let~$R\in k[T]_{\le d+1}$ be a monic polynomial of degree~$d+1$. An easy computation shows that $|R(x)| = \max(r,|(R-R_{0})(x_{d})|)$.  

Using $|(A-R_{0})(x)| = |(A-R_{0})(x_{d})|$, $|A(x)| < |A(x_{d})|$ and $|R_{0}(x)| \le |A(x)|$, we find $|R_{0}(x)| < |R_{0}(x_{d})|$. In particular, $R_{0}$~is irreducible. We want to prove that~$\eta_{R_{0},r}$ satisfies the property of the statement. 

Since~$\eta_{R_{0},r}$ and~$x_{d}$ are at least~$x$, they are comparable. From the inequality $|R_{0}(x_{d})| > r$, we deduce that $x_{d} \ge \eta_{R_{0},r} \ge x$. Hence we have $x_{d} = (\eta_{R_{0},r})_{d}$. Now, for every monic polynomial~$R$ of degree~$d+1$, we have
\[|R(\eta_{R_{0},r})| \le \max(r,|(R-R_{0})(\eta_{R_{0},r})|) = \max(r,|(R-R_{0})(x_{d})|) = |R(x)|\]
and we are done.
\end{proof}

The theorem enables us to define a map $x\in \Ak \mapsto x_{d} \in D_{d}$. It sends the point~$x$ to the closest point~$x_{d}$ on the graph~$D_{d}$. Remark that, for any~$d'\ge d$, we have $(x_{d'})_{d} = x_{d}$.

\begin{corollary}
Assume that~$k$ is maximally complete. The map
\[ \Ak \to \varprojlim_{d\ge 0} D_{d}\]
is a bijection.
\end{corollary}

\begin{remark}
It is also possible to describe the union of the sets~$D_{d}$. In the case where~$k$ is maximally complete, it contains all the points except some of type~1: those associated to the absolute values
\[P(T) \in k[T] \mapsto |P(\alpha)| \in \R_{+},\]
where~$\alpha$ is an element of~$\hat{\bar{k}}$, the completion of an algebraic closure of the field~$k$, which is transcendental over~$k$.
\end{remark}

\section{Approximation in higher dimensions}

In this second part, we would like to extend the approximation result of the previous section to the case of a general affine analytic space~$\E{n}{k}$, with $n\in\N$.

We will denote $\bT=(T_{1},\dots,T_{n})$ the variables on~$\E{n}{k}$. We endow $\N^n$ with the following order relation: $(d_{1},\dots,d_{n}) < (d'_{1},\dots,d'_{n})$ if 
\[{\renewcommand{\arraystretch}{1.8}\begin{array}{ll}
& \disp \sum_{i=1}^n d_{i} < \sum_{i=1}^n d'_{i}\\
\textrm{or} & \disp \sum_{i=1}^n d_{i} = \sum_{i=1}^n d'_{i} \textrm{ and } \exists j\in\cn{0}{n-1}, \forall i\le j, d_{i}=d'_{i} \textrm{ and } d_{j+1} < d'_{j+1}.
\end{array}}\]

Let $\bd \in\N^n$. We denote~$k[\bT]_{\le \bd}$ the $k$-vector space of polynomials of degree at most~$\bd$. Let~$R\subset \R_{+}$. We denote~$\Ss_{\bd,R}$ (or~$\Ss_{\bd}$ when $R=\R_{+}$) the set of subsets of~$\E{n}{k}$ of the form
\[\bigcap_{1\le i\le r} \{|P_{i}| \le r_{i}\},\]
with $P_{i} \in k[\bT]_{\le \bd}$ and~$r_{i}\in R$. 
%Let
%\[\Ss_{\bd} = \bigcup_{R \subset \R_{+}} \Ss_{\bd,R}.\]
%\[D_{R} = \bigcup_{\bd\in\N^n} \Ss_{\bd,R} \textrm{ and } D=\bigcup_{R\subset \R_{+}} D_{R}.\]
Let~$D_{\bd,R}$ be the union of the Shilov boundaries of elements of~$\Ss_{\bd,R}$.

If~$K$ is a non-empty compact subset of~$\E{n}{k}$, we let~$\|.\|_{K}$ denote the sup-norm on~$K$.

\begin{lemma} \label{lem:maxryhd}
Let~$K$ be a compact subset of~$\E{n}{k}$. Let~$\bd\in\N^n$ and assume that there exists a monic polynomial~$R_{0}$ of degree~$\bd$ such that, for any monic polynomial~$R$ of degree~$\bd$, we have $\|R_{0}\|_{K} \le \|R\|_{K}$.
Then, for any monic polynomial~$R$ of degree~$\bd$, we have
\[\|R\|_{K}= \max(\|R_{0}\|_{K}, \|R-R_{0}\|_{K}).\]
\end{lemma}

\begin{theorem}\label{thm:hd}
Assume that~$k$ is maximally complete. Let~$K$ be a non-empty compact subset of~$\E{n}{k}$. Let~$\bd\in\N^n$. There exists a unique set~$S_{K,\bd}$ in~$\Ss_{\bd}$ containing~$K$ such that   
\[\forall P\in k[\bT]_{\le \bd},\ \|P\|_{S_{K,\bd}}=\|P\|_{K}.\]

Moreover, we have
\[S_{K,\bd} = \bigcap_{P\in k[\bT]_{\le \bd}} \{|P|\le \|P\|_{K}\} \in \Ss_{\bd,R},\]
where $R \subset\R_{+}$ is the image of~$k[\bT]_{\le \bd}$ by the sup-norm~$\|.\|_{K}$.
\end{theorem}
\begin{proof}
The uniqueness and the last statement are clear. Let us prove the existence of~$S_{K,\bd}$ by induction on~$\bd$.  If $\bd=0$, choose $S_{0}=\E{n}{k}$.

Let~$\bd\in \N^n\setminus\{0\}$. Let~$\bd'$ be the predecessor of~$\bd$ and assume we found~$S_{K,\bd'}$ in~$\Ss_{K,\bd'}$ that satisfies the conditions of the statement. Let~$E$ be the separated completion of $k[\bT]_{\le \bd'}$ with respect to the semi-norm~$\|.\|_{K}$. It is a spherically complete $k$-vector space. Let $j : k[\bT]_{\le \bd'} \to E$ be the canonical morphism. 

The family of balls $\big(D_{P} = \overline{D}(j(P(\bT)), \|\bT^{\bd}-P(\bT)\|_{K})\big)_{P\in k[\bT]_{\le \bd'}}$ is nested. Let $j(P_{0}(\bT))$ be an element in the intersection. Let $R_{0}(\bT) =  \bT^{\bd}-P_{0}(\bT)$ and
\[r= \|R_{0}(\bT)\|_{K}= \inf_{P\in k[\bT]_{\le \bd'}} (\|\bT^{\bd}-P(\bT)\|_{K}).\]

Let us now set 
\[S_{K,\bd} = S_{K,\bd'} \cap \{|R_{0}|\le r\}.\] 
This is an element of~$\Ss_{\bd,R}$ that contains~$K$. For any polynomial~$P$ of degree smaller than~$\bd$, we have $\|P\|_{K}=\|P\|_{S_{K, \bd'}}=\|P\|_{S_{K,\bd}}$. Moreover, by lemma~\ref{lem:maxryhd}, for every monic polynomial~$R$ of degree~$\bd$, we have 
\[ \|R\|_{K} = \max(\|R_{0}\|_{K},\|R-R_{0}\|_{K}) = \max(r,\|R-R_{0}\|_{S_{K,\bd}}).\]
Using the fact that $\|R_{0}\|_{S_{K,\bd}}= r$, we deduce that $\|R\|_{K}=\|R\|_{S_{K,\bd}}$.
\end{proof}

\begin{remark}
As we found out after we wrote this paper, our construction is not unrelated to that of~\cite{HLen}, \S~5.1.
\end{remark}

\begin{corollary}\label{cor:approx}
Assume that~$k$ is maximally complete. Let~$K$ be a non-empty compact subset of~$\E{n}{k}$. Let~$\bd\in\N^n$ and $R \subset\R_{+}$ be the image of~$k[\bT]_{\le \bd}$ by the sup-norm~$\|.\|_{K}$. There exists a finite number of points~$\eta_{1},\dots,\eta_{r}$ in~$D_{\bd,R}$ such that  
\[\forall P\in k[\bT]_{\le \bd},\ \|P\|_{K} = \max_{1\le i\le r} (|P(\eta_{i})|).\]
\end{corollary}
\begin{proof}
Choose for $\eta_{1},\dots,\eta_{r}$ the points of the Shilov boundary of the $k$-affinoid space~$S_{K,\bd}$ of the theorem.
\end{proof}

\begin{remark}
If the compact~$K$ is a singleton~$\{x\}$, the result of theorem~\ref{thm:hd} in dimension~1 seems to be weaker than that of theorem~\ref{thm:dim1}, but it is actually easy to see that they are equivalent. As a matter of fact, the sets~$S_{x,d}$ are finite unions of virtual discs of the form~$\{|P|\le r\}$, where~$P$ is an irreducible factor of a polynomial that appears in the definition of~$S_{d}$. If~$x$ belongs to the disc~$\{|P|\le r\}$, then the point~$\eta_{P,r}$ satisfies the properties of theorem~\ref{thm:dim1}. 
\end{remark}

\begin{remark}
Even when~$K$ is a singleton~$\{x\}$, the generalisation to higher dimension needs to be handled with some care. For example, it is not true that the set~$S_{x,\bd}$ may be written in the form 
\[\{|P_{1}(T_{1})|\le r_{1}\} \cap \{|P_{1}(T_{1},T_{2})|\le r_{2}\} \cap \dots \cap \{|P_{1}(T_{1},\dots,T_{n})|\le r_{n}\}.\] 
Let us give an example. We will assume that the characteristic of~$k$ is not~2. Let $r,s \in [0,1)$ with $s < r^3$. Let us consider the point~$\eta_{r}$ in~$\Ak$ which is the unique point in the Shilov boundary of $\{|T_{1}|\le r\}$. Around this point, the fonction~$1+T_{1}$ has a square root given by the usual development $\sqrt{1+T_{1}} = 1+\frac{1}{2}\, T - \frac{1}{8}\, T^2 +\dots$ Let~$x$ in~$\E{2}{k}$ be the point in the fiber over~$\eta_{r}$ which is the unique point in the Shilov boundary of $\{|T_{2} - \sqrt{1+T_{1}}|\le s\}$. The subset~$S_{x,(2,0)}$ is not the naive
\[\{|T_{1}|\le r\} \cap \{|T_{2} - (1+\frac{1}{2}\, T - \frac{1}{8}\, T^2)|\le s\}\]
but
\[\{|T_{1}|\le r\} \cap \{|T_{2}^2 - (1+T_{1})|\le s\} \cap \{|T_{2}-1|\le s\}.\]
\end{remark}

\begin{remark}
It is not true either that the set~$S_{x,\bd}$ has a unique point in its Shilov boundary in general. Let us consider the point~$\eta_{1}$ in~$\Ak$ (with variable~$T_{1}$) and the point~$x$ in~$\E{3}{k}$ (with variables $T_{1},T_{2},T_{3}$) over it defined by the equations $T_{2}=T_{1}^4$ and $T_{3}=T_{1}^6$. It is possible to check that
\[{\renewcommand{\arraystretch}{1.3}\begin{array}{rcl}
S_{x,(3,0)} &=& \{|T_{1}|\le 1\} \cap \{|T_{2}|\le 1\} \cap \{|T_{3}|\le 1\}\\
&&  \cap\, \{T_{1}^2\, T_{2} - T_{3} = T_{1}^2\, T_{3} - T_{2}^2 = T_{2}^3 - T_{3}^2 = 0\}.
\end{array}}\]
This is a Zariski closed subset of the unit closed ball in~$\E{3}{k}$, which has two irreducible components (defined one by $T_{2}-T_{1}^4=T_{3}-T_{1}^6=0$, the other by $T_{2}=T_{3}=0$), hence two points in its Shilov boundary. We are grateful to N.~D.~Elkies, who kindly communicated this example to us (see~\cite{MOirrcomp}).
\end{remark}

\begin{corollary}
Assume that~$k$ is maximally complete. Let~$\bd\in\N^n$. Let~$\Ps$ be a family of polynomials of degrees at most~$\bd$. Let~$(r_{P})_{P\in\Ps}$ be a family of real numbers. Assume that the subset of~$\E{n}{k}$ defined by
\[K = \bigcap_{P\in\Ps} \{|P|\le r_{P}\}\]
is compact. Then it belongs to~$\Ss_{\bd}$.
\end{corollary}
\begin{proof}
We may assume that~$K$ is non-empty. It is easy to check that
\[K =  \bigcap_{Q\in k[\bT]_{\le \bd}} \{|Q|\le \|Q\|_{K}\}.\]
By theorem~\ref{thm:hd}, the latter set belongs to~$\Ss_{\bd}$.
\end{proof}

\begin{remark}\label{rem:dim2type4}
The result fails if the degrees are not bounded. Assume that~$k$ is not trivially valued. Let~$r>0$. Let~$(\alpha_{n})_{n\ge 0}$ be a sequence of elements of~$k$ such that the sequence of real numbers $(|\alpha_{n}|\, r^n)_{n\ge 0}$ is decreasing and has a positive limit. Let us consider the intersection
\[\{|T_{1}|\le r\} \cap \bigcap_{n\ge 1} \big\{|T_{2} - \sum_{i=0}^n \alpha_{i}\, T_{1}^{i}| \le |\alpha_{n+1}| r^{n+1}\big\}\]
in~$\E{2}{k}$. It is compact but not affinoid. For example, one may check that over~$\eta_{r}$, it contains only one point, which is of type~4 over~$\Hs(\eta_{r})$.
\end{remark}

For~$\bd\in\N^n$, let~$\Ss_{\bd}(k)$ be the family of subsets of~$k^n$ of the form 
\[\bigcap_{1\le i\le p} \{|P_{i}|\le r_{i}\},\] 
with $P_{1},\dots,P_{p} \in k[\bT]_{\le \bd}$ and $r_{1},\dots,r_{p}\in \R_{+}$.

\begin{corollary}
Assume that~$k$ is algebraically closed and maximally complete. Let~$\bd\in\N^n$. Let~$I$ be a linearly ordered set and $(V_{i})_{i\in I}$ be a decreasing family of non-empty elements of~$\Ss_{\bd}(k)$. Assume that there exists a disc $\overline{D}(0,\bs)$ with $\bs \in \R_{+}^n$ that contains all the~$V_{i}$'s. Then the intersection $\bigcap_{i\in I} V_{i}$ is non empty.
\end{corollary}
\begin{proof}
If~$k$ is trivially valued, it is a result about Zariski closed subsets of~$\A^n_{k}$ and we conclude by the noetherianity of the Zariski topology.

Let us now assume that~$k$ is not trivially valued. Let $\D = \overline{D}(0,\bs) \subset \E{n}{k}$. Let~$i\in I$. There exists an integer~$p_{i}$, polynomials $P_{i,1},\dots,P_{i,p_{i}}$ in $k[\bT]_{\le \bd}$ and real numbers $r_{i,1},\dots,r_{i,p_{i}}$ such that
\[V_{i} = \bigcap_{1\le j\le p_{i}} \big\{x\in k^n\, \big|\, |P_{i,j}|\le r_{i,j}\big\}.\]
Let us set
\[W_{i} = \bigcap_{1\le j\le p_{i}} \big\{x\in \E{n}{k}\, \big|\, |P_{i,j}|\le r_{i,j}\big\} \cap \D.\]

If~$J$ is a finite subset of~$I$, the intersection $\bigcap_{i\in J} W_{i}$ contains $\bigcap_{i\in J} V_{i}$, hence is non-empty. Since~$\D$ is compact, we deduce that $\bigcap_{i\in I} W_{i}$ is non-empty. By the preceeding corollary, it belongs to~$\Ss_{\bd}$. Since~$k$ is algebraically closed and not trivially valued, it has a $k$-point, which necessarily belongs to $\bigcap_{i\in I} V_{i}$.
\end{proof}

\begin{remark}
The result fails if we do not assume that the~$W_{i}$ belong to some disc (in dimension at least~2 and for non-trivially valued fields). Consider for instance $\{|T_{1}-1|\le r\} \cap \bigcap_{n\ge 0} \{|T_{1}T_{2}| \le r^n\}$, for some $r\in (0,1)$.
\end{remark}

\section{Definable types}

In this last part of the paper, we would like to apply some model-theoretic notions in the setting of Berkovich spaces. We refer to~\cite{Marker} for the basic definitions and results of model theory. We will use freely the rather intuitive notions of language, formula (built from the language possibly with variables and quantifiers) and sentence (formula with no free variable).

Since we work with Berkovich spaces, it is natural to consider valued fields $(k,|.|)$. From the model-theoretic point of view, we will use the language~$\Ls_{\div}$, which is the language of rings $(+,-,.,0,1)$ endowed with an extra binary predicate~``$\div$'', with $\div(x,y)$ interpreted as \mbox{$|x|\le |y|$}. We will work with the theory ``ACVF'' of algebraically closed valued fields with non-trivial valuation. One may check that the axioms needed may actually be written down in the language~$\Ls_{\div}$. Beware that there is no restriction on the rank of the valuation. For this reason, we will need to be careful when turning to Berkovich spaces. Remark that there is no way to get around this pitfall since valuations cannot be forced to have rank~1 by first-order formulas.

Let us recall one of the basic results of the theory, which is essentially due to A.~Robinson (see~\cite{complete}).

\begin{theorem}
\begin{enumerate}[\it i)]
\item The theory ACVF eliminates quantifiers: every formula is equivalent to a formula with no quantifiers.
\item Any embedding $K\subset L$ of algebraically closed valued fields with non-trivial valuations is elementary: a sentence with parameters in~$K$ holds in~$K$ if, and only if, it holds in~$L$. (The theory ACVF is said to be model-complete.)
\end{enumerate}
\end{theorem}

We plan to apply model-theoretic results to Berkovich spaces. In this setting, it is sometimes important to consider trivially valued fields, hence we would rather not exclude them in our model-theoretic arguments. It is of course possible to define the theory of algebraically closed trivially valued fields with the language~$\Ls_{\div}$. Let us remark that, in this case, the binary predicate~$\div$ is redondant since the formula $\div(x,y)$ is equivalent to $x=0 \vee (x\ne 0 \wedge y\ne 0)$. In other words, this theory is equivalent to the theory of algebraically closed fields, which is model-complete and eliminates quantifiers (see~\cite{Marker}, \S 3.2). 

In the sequel, we will allow ourselves to use model-completeness and quantifier elimination for algebraically closed valued fields, regardless of the fact that the valuation is trivial or not. Beware that this forces us to consider only trivially valued extensions of trivially valued fields.

\bigskip

We now introduce one more definition, that of a type. This model-theoretic notion is analogous to the notion of point of a Berkovich space. 

\begin{definition}
Let~$p$ be a set of formulas with parameters in $A\subset k$ and free variables from $\bu=(u_{1},\dots,u_{n})$. We say that~$p$ is a \textbf{type} if $p \cup \textrm{Th}_{A}(k)$ is satisfiable, where $\textrm{Th}_{A}(k)$ denotes the set of sentences with parameters in~$A$ that hold in~$k$. We say that a type~$p$ is \textbf{complete} if, for any formula~$\varphi(\bu)$, either $\varphi(\bu)$ or $\neg\varphi(\bu)$ belongs to~$p$.
\end{definition}

The basic example of type is constructed as follows. Let~$K$ be an algebraically closed extension of~$k$ (with trivial valuation if $k$~has trivial valuation). Let~$\ba\in K^n$. The set of formulas~$\varphi(\bu)$ with parameters in~$A\subset k$ such that $\varphi(\ba)$ holds in~$K$ is a complete type, denoted $\tp^K(\ba/A)$. It is also possible to associate a complete type to any point~$x$ of~$\E{n}{k}$ (either directly or by using the previous construction and considering $\tp^L(|T_{1}(x)|,\dots,|T_{n}(x)|/A)$, where~$L$ is any algebraically closed extension of~$\Hs(x)$). 

The converse statement is true, with some restrictions for the part concerning Berkovich spaces. General model-theoretic arguments (see~\cite{Marker}, corollary~4.1.4) show that, for any complete type~$p$ with parameters in~$A$, there exists an elementary extension~$K$ of~$k$ and a point~$\ba\in K^n$ such that $p=\tp^K(\ba/A)$. But it may well happen that the valuation of~$K$ has a rank greater than~1, which prevents it from defining a point in~$\E{n}{k}$. The generic type of the open unit ball gives such an example, since the valuation group of any of its realizations contains an infinitesimal below~1.

\begin{definition}
Let~$p$ be a type with parameters in $A\subset k$ and free variables from $\bu=(u_{1},\dots,u_{n})$. The type~$p$ is said to be \textbf{definable} over $B\subset k$ if, for every formula $\varphi(\bu,\bv)$ (without parameters), the exists a formula $d_{p}(\varphi)(\bv)$ with parameters in~$B$ such that, for any $\ba \in A^m$, the formula $\varphi(\bu,\ba)$ belongs to~$p$ if, and only if, $d_{p}(\varphi)(\ba)$ holds in~$k$.
\end{definition}

As we said above, points of affine Berkovich spaces give rise to types, but not always definable types. In~$\E{1}{k}$ with~$k$ algebraically closed and $|k|=\R_{+}$, for instance, definable types correspond to points of type~1, 2 or~3, but not~4 (consider the formula $\varphi(u,a,b): |u-a|\le |b|$ ; see~\cite{imaginaries}, lemma~2.3.8, for more on this question). Therefore, if~$k$ is algebraically closed and maximally complete and if $|k|=\R_{+}$, every point in~$\E{1}{k}$ corresponds to a definable type. But even over such a field, points of type~4 may appear again in higher dimensions. Indeed the residue field~$\Hs(x)$ of a point~$x$ of type~2 in~$\Ak$ fails to be maximally complete (in the non-trivially valued case). Hence, the fiber of the projection $\E{2}{k} \to \E{1}{k}$ over the point~$x$, which is isomorphic to~$\E{1}{\Hs(x)}$, contains points of type~4 (see remark~\ref{rem:dim2type4}).

In~\cite{HLen}, lemma~13.1.1, E.~Hrushovski and F.~Loeser explain that, over an algebraically closed and maximally complete field~$k$ with value group~$\R_{+}$, points of~$\E{n}{k}$, for any~$n$, correspond to (stably dominated) definable types. The proof is quite short but it makes use of the relation between types which are orthogonal to the value group and stably dominated types, as well as a criterion for a type to be stably dominated over an algebraically closed and maximally complete field, two difficult results from~\cite{domination}. As the reader may have guessed from the previous discussion, this result came to the author as a true surprise!

In what follows, we will try to explain in geometric terms what being a definable type means and recover E.~Hrushovski and F.~Loeser's result using our work from the previous sections. The rough idea is that points in~$D$ are definable types, which follows easily from quantifier elimination, and that other points are ``approximated'' by those (in the sense of corollary~\ref{cor:approx}). This approximation is good enough for type-definability. Indeed, to check it, we only need to consider one formula at a time and, any formula being finite, it is possible to bound the degree of the polynomials that appear in it.

\medskip

Let us first recall the definition of a semi-algebraic set (see~\cite{salg} d\'efinition~1.2).

\begin{definition}
A subset~$V$ of~$\E{n}{k}$ is called \textbf{semi-algebraic} (resp. \textbf{strictly semi-algebraic}) if it is a finite boolean combination of subsets of the form $\{|P| \le \lambda |Q|\}$ with $P,Q \in k[T]$ and~$\lambda\in \R_{+}$ (resp. $\lambda\in\{0,1\}$).
\end{definition}

\begin{theorem}\label{thm:semialgk}
Assume that~$k$ is maximally complete. Let~$n,m$ be integers. Let~$x$ be a point in~$\E{n}{k}$ such that $|\Hs(x)| = |k|$ and $V$~be a strictly semi-algebraic subset of~$\E{n+m}{k}$. Let $\pi : \E{n+m}{k} \to \E{m}{k}$ be the projection on the last~$m$ coordinates. Then the subset
\[ \{y \in\E{m}{k}(k)\, |\, x\in V\cap\pi^{-1}(y) \}\]
is strictly semi-algebraic.
\end{theorem}
\begin{proof}
Let us denote $(\bT,\bT')$ with~$\bT$ of length~$n$ and~$\bT'$ of length~$m$ the variables on~$\E{n+m}{k}$. The formula defining~$V$ may be constructed using conjonctions, disjonctions and negations of atomic formulas of the form $|Q(\bT,\bT')| \le |R(\bT,\bT')|$. Let~$\bd$ be the maximum of the degrees in~$\bT$ of the polynomials appearing in the formula. By theorem~\ref{thm:hd}, there exists polynomials $P_{1},\dots,P_{t} \in k[\bT]$ and real numbers $r_{1},\dots,r_{t} \in |\Hs(x)|=|k|$ such that, if we denote
\[ S = \bigcap_{1\le i\le t} \{|P_{i}(\bT)|\le r_{i}\},\]
we have
\[\forall P\in k[\bT]_{\le \bd},\ \|P\|_{S}= |P(x)|.\]
Remark that we may assume that all the~$r_{i}$ are~$0$ or~$1$. Let~$K/k$ be an extension of valued fields with~$K$ algebraically closed (and trivially valued in case~$k$ is) such that
\[\forall P\in k[\bT], \exists \ba\in K^n, \|P\|_{S} = |P(\ba)|.\]
(One may for example choose an algebraic closure of~$\Hs(x)$ and let~$\ba$ be the image of~$\bT$ by the evaluation morphism $k[\bT]\to \Hs(x) \hookrightarrow K$.)

Let~$\by \in k^m$. It is easy to check that we have $|Q(x,\by)| \le |R(x,\by)|$ if, and only if,
\[ \forall \alphab\in S(K), \exists \betab\in S(K), |Q(\alphab,\by)|  \le |R(\betab,\by)|.\] 
Since~$S$ is definable over~$k$, this proves that the subset of the statement is definable over~$k$, hence strictly semi-algebraic by quantifier elimination.
\end{proof}

\begin{remark}\label{rem:typedef}
When~$k$ is algebraically closed, the statement exactly means that the type~$p_{x}$ corresponding to the point~$x$ is definable (over~$k$). Let~$\varphi(\bu,\bv)$ be a formula with~$\bu$ and~$\bv$ of lengths~$n$ and~$m$. We want to find a formula~$d_{x}(\varphi)(\bv)$ such that, for any $\by\in k^m$, $\varphi(\bu,\by)$ belongs to~$p_{x}$ if, and only if, $d_{x}(\varphi)(\by)$ holds in~$k$. By quantifier elimination, we may assume that~$\varphi$ is a quantifier-free formula. In this case, it defines a strictly semi-algebraic subset~$V$ of~$\E{m+n}{k}$. We now are in the situation of the theorem and any formula defining \mbox{$\{\by \in k^m\, |\, x\in V\cap\pi^{-1}(\by)\}$} will do.
\end{remark}

The fact that a point over~$k$ corresponds to a definable type will enable us to extend it canonically over extensions~$K/k$.

\begin{corollary}\label{cor:extensionpoint}
Assume that~$k$ is maximally complete and algebraically closed. Let~$n$ be an integer and~$x$ be a point in~$\E{n}{k}$ such that $|\Hs(x)| = |k|$. Let $(S_{\bd})_{\bd\in\N^n}$ be the family of subsets associated to~$x$ by theorem~\ref{thm:hd}. 

Let~$K/k$ be an extension of complete valued fields. If~$k$ is trivially valued, assume that $K$~is trivially valued too. Let $\pi : \E{n}{K} \to \E{n}{k}$ be the natural morphism. There exists a unique point $\sigma_{K}(x) \in \E{n}{K}$ such that 
\[\forall\bd\in\N^n, \forall P\in K[\bT]_{\le \bd},\ \|P\|_{\pi^{-1}(S_{\bd})}= |P(\sigma_{K}(x))|.\]

Moreover, the point~$\sigma_{K}(x)$ may also be characterized by the following property:
\[\forall P\in K[\bT],\ |P(\sigma_{K}(x))| = \|P\|_{\pi^{-1}(x)}.\]
\end{corollary}
\begin{proof}
We may assume that~$K$ is algebraically closed. Then it is an elementary extension of~$k$. The point~$x$ corresponds to a type~$p_{x}$ over~$k$, which is definable by remark~\ref{rem:typedef}.  We define a type~$q=p_{x}|K$ over~$K$ using the following recipe (and the notations of the remark). For any~$\bv \in K^m$, the formula $\varphi(\bu,\bv)$ belongs to~$q$ exactly when $d_{x}(\varphi)(\bv)$ holds in~$K$. To check that this indeed defines a type over~$K$, we only need to check that finite subsets of formulas of~$q$ are consistent with the theory of~$K$. Let $\varphi_{1}(\bu,\bv_{1}), \dots, \varphi_{r}(\bu,\bv_{r}) \in q$. By definition of the~$d_{x}(\varphi)$'s and the fact that~$p_{x}$ is a type, the sentence
\[\forall w_{1},\dots,w_{r}, \bigwedge_{1\le i\le r} d_{x}(\varphi_{i})(w_{i}) \to \exists \bu, \bigwedge_{1\le i\le r} \varphi_{i}(\bu,w_{i})\]
holds in an elementary extension of~$k$, hence in~$k$ and~$K$.

Since~$|\Hs(x)|=|k|$, for every~$\bd\in\N^n$, the sentence~$\psi_{\bd}(\bu)$ defined by 
\[\forall \textrm{ polynomial } P \textrm{ of degree } \bd, \exists v, |P(\bu)|=|v|\]
belongs to~$p_{x}$. Hence it also belongs to~$q$, which means that~$q$ corresponds to a point of~$\E{n}{K}$. We will denote it~$\sigma_{K}(x)$.

Let~$\bd\in \N^n$. There exists polynomials $P_{1},\dots,P_{t} \in k[\bT]_{\bd}$ and real numbers $r_{1},\dots,r_{t} \in \{0,1\}$ such that
\[ S_{\bd} = \bigcap_{1\le i\le t} \{|P_{i}|\le r_{i}\}.\]

Let us consider a formula~$\varphi(\bu,\bv)$ expressing the fact that the sup-norm on~$S_{\bd}$ of a polynomial~$P$ of degree at most~$\bd$ is realized at the point we consider. More precisely, the formula~$\varphi(\bu,\bv)$ looks like
\[ \begin{array}{l}
\forall \textrm{ polynomial } P \textrm{ of degree at most } \bd,\\
\big(\forall \ba, |Q_{1}(\ba)| \le r_{1} \wedge \dots \wedge  |Q_{t}(\ba)| \le r_{t} \to |P(\ba)| \le |P(\bu)| \big)\\
\wedge \big( \exists \ba,  |Q_{1}(\ba)| \le r_{1} \wedge \dots \wedge  |Q_{t}(\ba)| \le r_{t} \wedge |P(\ba)| = |P(\bu)| \big),
\end{array}\]
where~$\bv$ encodes the coefficients of the~$Q_{i}$'s. 

When~$\bv$ is specialized to the tuple~$\bv_{0}$ that corresponds to the coefficients of the~$P_{i}$'s, the formula $\varphi(\bu,\bv_{0})$ belongs to the type~$p_{x}$, so $d_{\varphi}(\bv_{0})$ holds in~$k$, hence in~$K$. By definition, $\varphi(\bu,\bv_{0})$ then belongs to the type $q=p_{x}|K$, which exactly means that, for any polynomial~$P$ of degree at most~$\bd$, $|P(\sigma_{K}(x))|$ is equal to the sup-norm of~$P$ on~$\pi^{-1}(S_{\bd})$ (which is defined by the same inequalities as~$S_{\bd}$). We have now proved the existence of a point~$x_{K}$ satisfying the required conditions. Its uniqueness is clear, as well as the last statement.
\end{proof}

\begin{remark}
In~\cite{rouge}, \S~5.2, V.~Berkovich introduced a notion of ``peaked point'' for points that may be canonically lifted over extensions of the base field. (The definition is actually given in terms of norms and is quite close to the last condition of corollary~\ref{cor:extensionpoint}.) We studied this notion further in~\cite{Angieen}, under the name of ``universal point'', and proved, in particular, that any point over an algebraically closed valued field is universal (see corollaire~3.14). The interested reader will find properties of the map~$\sigma_{K}$ as well as applications of the notion in those two references.
\end{remark}

\begin{theorem}
Assume that~$k$ is maximally complete, algebraically closed and non-trivially valued. Let $n,m$ be integers. Let~$x$ be a point in~$\E{n}{k}$ such that $|\Hs(x)| = |k|$ and $V$~be a (resp. strictly) semi-algebraic subset of~$\E{n+m}{k}$. Let $\pi : \E{n+m}{k} \to \E{m}{k}$ be the projection on the last~$m$ coordinates. Then the subset
\[\{y\in\E{m}{k}\, |\, \sigma_{\Hs(y)}(x)\in V\cap\pi^{-1}(y)\}\]
is (resp. strictly) semi-algebraic.
\end{theorem}
\begin{proof}
When~$V$ is strictly semi-algebraic, the proof goes exactly like the proof of theorem~\ref{thm:semialgk}, because corollary~\ref{cor:extensionpoint} enables us to use the same~$S_{\bd}$'s to define the point~$\sigma_{\Hs(y)}(x)$ in each fiber~$\pi^{-1}(y)$.

To handle the general case, we now use the same trick as in the proof of~\cite{salg}, proposition~1.6. We may assume that~$V$ is defined by a formula obtained by conjonctions, disjonctions and negations of atomic formulas of the form $|P_{i}(\bT,\bT')| \le \lambda_{i} |Q_{i}(\bT,\bT')|$ for $i\in\cn{1}{p}$. Let us now consider the space~$\E{n+m+p}{k}$ and denote $S_{1},\dots,S_{p}$ the $p$~last variables. Define a stricly semi-algebraic subset~$W$ of~$\E{n+m+p}{k}$ by replacing, for $i\in\cn{1}{p}$, the $i^\textrm{th}$~atomic formula by $|P_{i}(\bT,\bT')| \le |S_{i}\, Q_{i}(\bT,\bT')|$.

Let us also denote $\pi : \E{n+m+p}{k} \to \E{m+p}{k}$ the projection on the last~$m+p$ coordinates. By the strictly algebraic case, the set $W' = \{z\in\E{m+p}{k}\, |\, \sigma_{\Hs(z)}(x) \in W\cap\pi^{-1}(z)\}$ is strictly semi-algebraic. Let $y\in\E{m}{k}$. Let~$y_{\lambdab}$ be the point of~$\E{m+p}{k}$ associated to the semi-norm
\[P(\bT',\bS) = \sum_{\bu\ge 0} p_{j}(\bT') \bS^\bu \mapsto \max_{\bu\ge 0}(|p_{j}(y)| \lambdab^\bu).\]
The subset $W\cap\pi^{-1}(y_{\lambdab})$ is a semi-algebraic subset of~$\E{n}{\Hs(y_{\lambdab})}$ defined by a boolean combination of the formulas $|P_{i}(\bT,y)| \le \lambda_{i}\, |Q_{i}(\bT,y)|$ for $i\in\cn{1}{p}$. It contains the point~$\sigma_{\Hs(y_{\lambdab})}(x)$ exactly when the subset $V\cap\pi^{-1}(y)$ of~$\E{n}{\Hs(y)}$ contains the point~$\sigma_{\Hs(y)}(x)$. So the set~$V'$ of the statement is equal to $\{y\in \E{m}{k}\, |\, y_{\lambdab} \in W'\}$.

If $|A(\bT',\bS)| \le |B(\bT',\bS)|$, with $A = \sum_{\bu\ge 0} a(\bT') \bS^\bu$ and $B = \sum_{\bv\ge 0} b(\bT') \bS^\bv$ (each of the sum actually containing a finite number of non-zero terms), is one of the inequations defining~$W'$, we see that $|A(y_{\lambdab})| \le |B(y_{\lambdab})|$ if, and only if, $\max_{\bu\ge 0} (|a(y)|\lambdab^\bu) \le \max_{\bv\ge 0} (|b(y)|\lambdab^\bv)$. This proves that~$V'$ is semi-algebraic. 
\end{proof}

\nocite{}
\bibliographystyle{alpha}
\bibliography{biblio}

\end{document}